\documentclass[11pt]{article}
\usepackage{mathrsfs}
\usepackage{bbm}
\usepackage{amsfonts}
\usepackage{amsmath,amssymb}
\usepackage{amsmath}
\usepackage{amssymb}
\usepackage{graphicx}

\newtheorem{theorem}{Theorem}[section]
\newtheorem{corollary}[theorem]{Corollary}

\newtheorem{proposition}[theorem]{Proposition}

\newenvironment{proof}[1][Proof]{\noindent \textbf{#1.} }{\  $\Box$}
\numberwithin{equation}{section}

\begin{document}

\title{\textbf{A Note On  $G$-normal Distributions}}
\author{ Yongsheng Song\thanks{Academy of Mathematics
and Systems Science, CAS, Beijing, China, yssong@amss.ac.cn.
Research supported by by NCMIS; Youth Grant of NSF (No. 11101406);
Key Project of NSF (No. 11231005); Key Lab of Random Complex
Structures and Data Science, CAS (No. 2008DP173182).} }
\maketitle
\date{}

\begin{abstract} As is known, the convolution $\mu*\nu$ of two $G$-normal distributions $\mu, \nu$ with different intervals of variances  may not be $G$-normal. We shows that $\mu*\nu$ is  a $G$-normal distribution if and only if $\frac{\overline{\sigma}_\mu}{\underline{\sigma}_\mu}=\frac{\overline{\sigma}_\nu}{\underline{\sigma}_\nu}$.

\end{abstract}


\textbf{MSC-classification}: 60E10


\section{Introduction}
Peng (2007) introduced the notion of $G$-normal distribution via the viscosity solutions of the $G$-heat equation below
\begin {eqnarray*}
\partial_t u-G(\partial^2_x u)&=&0, \ (t,x)\in (0,\infty)\times \mathbb{R},\\
                        u(0,x)&=& \varphi (x),
\end {eqnarray*} where $G(a)=\frac{1}{2}(\overline{\sigma}^2 x^+-\underline{\sigma}^2 x^-)$, $a\in \mathbb{R}$ with $0\leq\underline{\sigma}\leq\overline{\sigma}<\infty$, and $\varphi\in C_{b,Lip}(\mathbb{R})$, the collection of bounded Lipstchiz functions on $\mathbb{R}$.

 Then the one-dimensional $G$-normal  distribution is defined by \[N_G[\varphi]=u^\varphi(1,0),\] where $u^\varphi$ is the  viscosity solution to the $G$-heat equation with the initial value $\varphi$. We denote by $\mathcal{G}^0$ the collection of functions $G$ defined above. For $G_1, G_2\in \mathcal{G}^0$,  the convolution between $N_{G_1}, N_{G_2}$ is defined as $N_{G_1}*N_{G_2}[\varphi]:=N_{G_1}[\phi]$ with $\phi(x):=N_{G_2}[\varphi(x+\cdot)]$.

 As is well-known, the convolution of two normal distributions is also   normal. How about $G$-normal distributions?
 We state the question in the PDE language: For $G_1, G_2\in \mathcal{G}^0$, set $G(t,a)=G_2(a), \ t\in(0,\frac{1}{2}]$ and $G(t,a)=G_1(a), \ t\in(\frac{1}{2},1]$. Let $v^\varphi$ be the viscosity solution to the following PDE
 \begin {eqnarray*}
\partial_t v-G(t,\partial^2_x v)&=&0, \\
                        v(0,x)&=& \varphi (x).
\end {eqnarray*} Can we find a function $G\in \mathcal{G}^0$ such that \[v^\varphi(1,0)=u^\varphi(1,0), \ \textmd{for all} \ \varphi\in C_{b,Lip}(\mathbb{R}) ?\] Here $u^\varphi$ is the  viscosity solution to the $G$-heat equation. However, \cite {H10} gave a counterexample to show that generally $N_{G_1}*N_{G_2}$ is not $G$-normal any more.

We denote by $\mathcal{G}$ the subset of $\mathcal{G}^0$ consisting of the elements with $\underline{\sigma}>0$. $N_G$ is called non-degenerate if $G$ belongs to $\mathcal{G}$.   For $G\in \mathcal{G}$, set \[\beta_G:=\frac{\overline{\sigma}}{\underline{\sigma}} \ \textmd{ and} \ \sigma_G=\frac{\overline{\sigma}+\underline{\sigma}}{2}. \] For abbreviation, we write $\beta, \sigma$ instead of $\beta_G, \sigma_G$ when no  confusion can arise.

We shall prove that $N_{G_1}*N_{G_2}$ is  a $G$-normal distribution if and only if $\beta_{G_1}=\beta_{G_2}$. Besides, the convolution between $G$-normal distributions is not commutative. We also prove that $N_{G_1}*N_{G_2}=N_{G_2}*N_{G_1}$ if and only if $\beta_{G_1}=\beta_{G_2}$.

Therefore, in order to emphasize the importance of the ratio $\beta$, we denote the $G$-normal distribution $N_G$ by $N_{\beta}(0, \sigma^2)$.

\section{Characteristic Functions}
First we shall consider the solutions of special forms to the $G$-heat equation
\begin {eqnarray}\label {G}
\partial_t u-G(\partial^2_x u)=0.
\end {eqnarray} Assume that $u(t,x)=a(t)\phi(x)$ with $a(t)\geq 0$ is a solution to the $G$-heat equation (\ref {G}). Then we conclude that
\[\frac{a'(t)}{a(t)}=\frac{G(\phi''(x))}{\phi(x)}\] is a constant. Assuming  that they are equal to $-\frac{\rho^2}{2}$, we have $a(t)=e^{-\frac{\rho^2}{2}t}$ and

\begin {eqnarray}\label {C}
G(\phi''(x))=-\frac{\rho^2}{2}\phi(x).
\end {eqnarray}
If $\phi''(x)$ is positive, the equation ({\ref {C}}) reduces to
\begin {eqnarray}\label {PC}
\overline{\phi}''(x)=-\frac{\rho^2}{\overline{\sigma}^2}\overline{\phi}(x).
\end {eqnarray} and the solution is \[\overline{\phi}(x)=\overline{\lambda} \cos(\frac{\rho}{\overline{\sigma}}x+\overline{\theta}).\]
On the other hand, if $\phi''(x)$ is negative, the equation ({\ref {C}}) reduces to
\begin {eqnarray}\label {NC}
\underline{\phi}''(x)=-\frac{\rho^2}{\underline{\sigma}^2}\underline{\phi}(x).
\end {eqnarray} and the solution is \[\underline{\phi}(x)=\underline{\lambda} \cos(\frac{\rho}{\underline{\sigma}}x+\underline{\theta}).\]
Motivated by the arguments above, we shall construct the solutions to the equation (\ref {C}) in the following way. Denoting by $\beta=\frac{\overline{\sigma}}{\underline{\sigma}}$, set
\begin {equation}\phi_\beta(x)=
\begin {cases}\frac{2}{1+\beta}\cos (\frac{1+\beta}{2}x) & \textmd{for $x\in[-\frac{\pi}{1+\beta},\frac{\pi}{1+\beta})$;}\\
\frac{2\beta}{1+\beta}\cos (\frac{1+\beta}{2\beta}x+\frac{\beta-1}{2\beta}\pi) & \textmd{for $x\in[\frac{\pi}{1+\beta},\frac{(2\beta+1)\pi}{1+\beta})$.}
\end {cases}
\end {equation}  This is a variant of the trigonometric function $\cos x$ (see Figure 1).

\begin {figure} [!ht]
\centering
\includegraphics[scale=1]{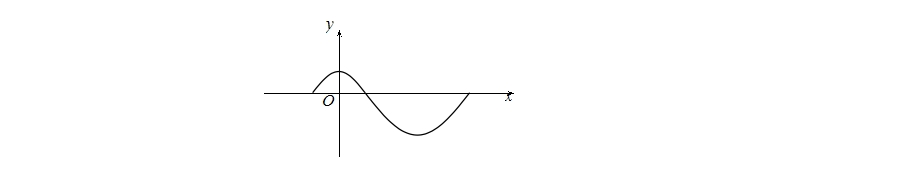}
\caption {$\phi_\beta(x)$}
\end {figure}

Then extend the definition of $\phi_\beta$ to the whole real line by the property $\phi(x+2k\pi)=\phi(x)$, $k\in \mathbb{Z}$. Clearly,  $\phi_1(x)=\cos x$ and $\phi_\beta$ belongs to $C^{2,1}(\mathbb{R})$, the space of bounded functions on $\mathbb{R}$ with uniformly Lipschitz continuous second-order derivatives.

\begin {proposition} $\phi_\beta$ is a solution to equation (\ref {C}) with $\rho=\frac{\underline{\sigma}+\overline{\sigma}}{2}=:\sigma$.
\end {proposition}
\begin {proof} The proof follows from simple calculations. For $x\in[-\frac{\pi}{1+\beta},\frac{\pi}{1+\beta})$,
\[\phi''_\beta(x)=-\frac{1+\beta}{2}\cos (\frac{1+\beta}{2}x)=-\frac{(1+\beta)^2}{4}\phi_\beta(x)\leq 0.\] So
\[G(\phi''_\beta(x))=-\frac{(1+\beta)^2\underline{\sigma}^2}{8}\phi_\beta(x)=-\frac{\sigma^2}{2}\phi_\beta(x).\] For $x\in[\frac{\pi}{1+\beta},\frac{(2\beta+1)\pi}{1+\beta})$,
\[\phi''_\beta(x)=-\frac{1+\beta}{2\beta}\cos (\frac{1+\beta}{2\beta}x+\frac{\beta-1}{2\beta}\pi)=-\frac{(1+\beta)^2}{4\beta^2}\phi_\beta(x)\geq 0.\] So
\[G(\phi''_\beta(x))=-\frac{(1+\beta)^2\overline{\sigma}^2}{8\beta^2}\phi_\beta(x)=-\frac{\sigma^2}{2}\phi_\beta(x).\]

\end {proof}
\begin {corollary} Let $\sigma:=\frac{\underline{\sigma}+\overline{\sigma}}{2}$ and $\beta:=\frac{\overline{\sigma}}{\underline{\sigma}}$. $e^{-\frac{{\sigma}^2}{2}t}\phi_\beta(x)$ is the solution to equation (\ref {G}) with the initial value $\phi_\beta(x)$.
In other words, $N_G[\phi_\beta(x+\sqrt{t}\cdot)]=e^{-\frac{{\sigma}^2}{2}t}\phi_\beta(x)$.
\end {corollary}
For $\lambda>0$ and  $c, \theta \in R$, set $\phi_{\beta}^{ \lambda, c, \theta}(x):=\lambda \phi_{\beta}(cx+\theta)$. It's easy to check that \[G((\phi_{\beta}^{ \lambda, c, \theta})'')=-\frac{c^2{\sigma}^2}{2}\phi_{\beta}^{ \lambda, c, \theta}.\] So $e^{-\frac{c^2{\sigma}^2}{2}t}\phi_{\beta}^{ \lambda, c, \theta}(x)$ is the solution to equation (\ref {G}) with the initial  value
$\phi_{\beta}^{ \lambda, c, \theta}(x)$. For any $\beta>1$, we call \[(\phi_{\beta}^{ \lambda, c, \theta}(x))_{\lambda,c,\theta}\] the characteristic functions of $G$-normal distributions $N_\beta(0, \sigma^2)$.

\section{Cramer's Type Decomposition  for $G$-normal Distributions}
Let's introduce more properties on the characteristic functions $\{\phi_\beta(x)\}_{\beta\geq1}$. For $1\leq\alpha <\beta<\infty$, we have \[\phi_\alpha(x)-\phi_\beta(x)\geq\frac{2(\beta-\alpha)}{(1+\alpha)(1+\beta)}=:e(\alpha,\beta)>0.\]

\begin {figure} [!ht] \label {f2}
\centering
\includegraphics[scale=1]{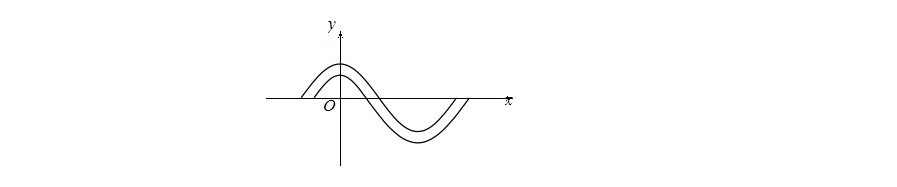}
\caption {$\phi_\alpha(x)$ and $\phi_\beta(x)$}
\end {figure}

\begin {theorem} For $G_1, G_2\in \mathcal{G}$, the convolution $N_{G_1}*N_{G_2}$ is a $G$-normal distribution if and only if $\beta_{G_1}:=\frac{\overline{\sigma}_{G_1}}{\underline{\sigma}_{G_1}}=\frac{\overline{\sigma}_{G_2}}{\underline{\sigma}_{G_2}}=:\beta_{G_2}$.
\end {theorem}
\begin {proof} The sufficiency is obvious, and we only prove the necessity.  Assume $N:=N_{G_1}*N_{G_2}$ is a $G$-normal distribution. Then $\underline{\sigma}_N^2=\underline{\sigma}_{G_1}^2+\underline{\sigma}_{G_2}^2$ and $\overline{\sigma}_N^2=\overline{\sigma}_{G_1}^2+\overline{\sigma}_{G_2}^2$. If $\beta_{G_1}\neq\beta_{G_2}$, we may assume that $\beta_{G_1}<\beta_{G_2}$, and the other case can be proved similarly. Then we have $\beta_{G_1}<\beta_N<\beta_{G_2}$. On one hand we have
\[N[\phi_{\beta_N}(\sqrt{t}\cdot)]=e^{-\frac{\sigma^2_N}{2}t}\phi_{\beta_N}(0)=\frac{2}{1+\beta_N}e^{-\frac{\sigma^2_N}{2}t}.\] On the other hand we have
\begin {eqnarray*}& &N_{G_2}[\phi_{\beta_N}(\sqrt{t}(x+\cdot))]\\
     &\geq&N_{G_2}[\phi_{\beta_{G_2}}(\sqrt{t}(x+\cdot))]+e(\beta_N,\beta_{G_2})\\
     &=&e^{-\frac{\sigma^2_{G_2}}{2}t}\phi_{\beta_{G_2}}(\sqrt{t}x)]+e(\beta_N,\beta_{G_2})\\
     &\geq&- \frac{2\beta_{G_2}}{1+\beta_{G_2}}e^{-\frac{\sigma^2_{G_2}}{2}t}+e(\beta_N,\beta_{G_2}).
\end {eqnarray*} Consequently, \[N[\phi_{\beta_N}(\sqrt{t}\cdot)]=N_{G_1}*N_{G_2}[\phi_{\beta_N}(\sqrt{t}\cdot)]\geq - \frac{2\beta_{G_2}}{1+\beta_{G_2}}e^{-\frac{\sigma^2_{G_2}}{2}t}+e(\beta_N,\beta_{G_2}).\]

Combining the above arguments we have
   \[\frac{2}{1+\beta_N}e^{-\frac{\sigma^2_N}{2}t}\geq - \frac{2\beta_{G_2}}{1+\beta_{G_2}}e^{-\frac{\sigma^2_{G_2}}{2}t}+e(\beta_N,\beta_{G_2}), \ \textmd{for any} \ t>0,\] which is a contradiction noting that $e(\beta_N,\beta_{G_2})>0$.
\end {proof}

\begin {theorem} For $G_1, G_2\in \mathcal{G}$,  $N_{G_1}*N_{G_2}=N_{G_2}*N_{G_1}$  if and only if $\beta_{G_1}=\beta_{G_2}$.
\end {theorem}
\begin {proof} We shall only prove the necessity. Assume that $\beta_{G_1}<\beta_{G_2}$. Then \[N_{G_1}[\phi_{\beta_{G_1}}(\sqrt{t}x+\sqrt{t}\cdot)]=e^{-\frac{\sigma^2_{G_1}}{2}t}\phi_{\beta_{G_1}}(\sqrt{t}x)\leq \frac{2}{1+\beta_{G_1}} e^{-\frac{\sigma^2_{G_1}}{2}t}.\] So \[N_{G_2}*N_{G_1}[\phi_{\beta_{G_1}}(\sqrt{t}\cdot)]\leq \frac{2}{1+\beta_{G_1}} e^{-\frac{\sigma^2_{G_1}}{2}t}.\] On the other hand
\begin {eqnarray*}& &N_{G_2}[\phi_{\beta_{G_1}}(\sqrt{t}x+\sqrt{t}\cdot)]\\
     &\geq&N_{G_2}[\phi_{\beta_{G_2}}(\sqrt{t}x+\sqrt{t}\cdot)]+e(\beta_{G_1},\beta_{G_2})\\
     &=&e^{-\frac{\sigma^2_{G_2}}{2}t}\phi_{\beta_{G_2}}(\sqrt{t}x)+e(\beta_{G_1},\beta_{G_2})\\
     &\geq&- \frac{2\beta_{G_2}}{1+\beta_{G_2}}e^{-\frac{\sigma^2_{G_2}}{2}t}+e(\beta_{G_1},\beta_{G_2}),
\end {eqnarray*}which implies that
\[N_{G_1}*N_{G_2}[\phi_{\beta_{G_1}}(\sqrt{t}\cdot)]\geq - \frac{2\beta_{G_2}}{1+\beta_{G_2}}e^{-\frac{\sigma^2_{G_2}}{2}t}+e(\beta_{G_1},\beta_{G_2}).\] Noting that $e(\beta_{G_1},\beta_{G_2})>0$, we have \[N_{G_1}*N_{G_2}[\phi_{\beta_{G_1}}(\sqrt{t}\cdot)]>N_{G_2}*N_{G_1}[\phi_{\beta_{G_1}}(\sqrt{t}\cdot)] \ \textmd{for $t$ large enough.}\]
\end {proof}


\renewcommand{\refname}{\large References}{\normalsize \ }

\end{document}